\documentclass[11p,leqno]{amsart}
\usepackage{amsmath}
\usepackage{amsfonts}
\usepackage{amssymb}
\usepackage{enumitem}
\usepackage{graphicx}
\usepackage{hyperref}
\usepackage{url}
\usepackage[backrefs]{amsrefs}
\usepackage{color}
\newdimen\AAdi%
\newbox\AAbo%
\def\AAk#1#2{\s_etbox\AAbo=\hbox{#2}\AAdi=\wd\AAbo\kern#1\AAdi{}}%
\def\AAr#1#2#3{\s_etbox\AAbo=\hbox{#2}\AAdi=\ht\AAbo\raise#1\AAdi\hbox{#3}}%
\font\tenmsb=msbm10 at 11pt \font\sevenmsb=msbm7 at 8pt
\font\fivemsb=msbm5 at 6pt
\newfam\msbfam
\textfont\msbfam=\tenmsb \scriptfont\msbfam=\sevenmsb
\scriptscriptfont\msbfam=\fivemsb

\renewcommand{\theequation}{\thesection\arabic{equation}}

\font\bbbld=msbm10 scaled\magstephalf

\newcommand{\ba}{\begin{array}}
\newcommand{\ea}{\end{array}}

\newcommand{\bfR}{\hbox{\bbbld R}}

\newcommand{\bfS}{\mathbb{S}}
\newcommand{\E}{\mathrm e}

\def\aint{\frac{\ \ }{\ \ }{\hskip -0.4cm}\int}

\newtheorem{theorem}{Theorem}[section]
\newtheorem{lemma}[theorem]{Lemma}
\newtheorem{proposition}[theorem]{Proposition}
\newtheorem{corollary}[theorem]{Corollary}

\theoremstyle{definition}

\theoremstyle{remark}
\newtheorem{remark}[theorem]{Remark}

\begin{document}
\setlength{\baselineskip}{1.2\baselineskip}

\title[Flow by $K^\alpha$]{\bf Flow by powers of the Gauss curvature}

\author{Ben Andrews}
\address{Mathematical Sciences Institute, Australia National University, ACT 2601 Australia}
\email{Ben.Andrews@anu.edu.au}

\author{Pengfei Guan}
\address{Department of Mathematics and Statistics\\
        McGill University\\
        Montreal, Quebec, H3A 2K6, Canada.}
\email{guan@math.mcgill.ca}
\author{Lei Ni}
\address{Department of Mathematics, University of California at San Diego, La Jolla, CA 92093, USA}
\email{lni@math.ucsd.edu}
\thanks{The first author was supported by Discovery Projects grant DP120102462 and Australian Laureate Fellowship FL150100126 of
 the Australian Research Council.  The research of the second author is partially supported by an NSERC Discovery Grant, and the research of the third author is partially supported by NSF grant DMS-1401500.}
\begin{abstract}
We prove that convex hypersurfaces in ${\mathbb R}^{n+1}$ contracting under the flow by any power $\alpha>\frac{1}{n+2}$ of the Gauss curvature converge (after rescaling to fixed volume) to a limit which is a smooth, uniformly convex self-similar contracting solution of the flow. Under additional central symmetry of the initial body we prove that the limit is the round sphere.
\end{abstract}

\maketitle

\leftskip 0 true cm \rightskip 0 true cm
\renewcommand{\theequation}{\arabic{equation}}
\setcounter{equation}{0} \numberwithin{equation}{section}

\section{Introduction}

 In this paper we study the flow of convex hypersurfaces $\tilde X(\cdot, \tau): M\to \mathbb{R}^{n+1}$ by the $\alpha$-power of Gauss curvature:
\begin{equation}\label{eq-gcf-alpha}
\frac{\partial}{\partial \tau}\tilde X(x, \tau)= -\tilde K^{\alpha}(x, \tau)\, \nu(x, \tau).
\end{equation}
Here $\nu(x, \tau)$ is the unit exterior normal at $\tilde X(x, \tau)$ of $\tilde M_\tau=\tilde X(M, \tau)$, and $\tilde K(x,\tau)$ is the Gauss curvature of $\tilde M_\tau$ at $\tilde X(x,\tau)$ (the tildes distinguish these from the normalized counterparts introduced below).

Equation (\ref{eq-gcf-alpha}) is a parabolic fully nonlinear equation of Monge-Amper\'e type, hence the study sheds light on the general theory of such equations.  The case $\alpha=1$ was proposed by Firey \cite{Firey} as a model for the wearing of tumbling stones.  The equation with general powers also arises in the study of affine geometry and of image analysis \cite{ST1, ST2, AGLM, CS, NK}.  For large $\alpha$ the equation becomes more degenerate and for small $\alpha$ it becomes more singular. Studying them together gives an example of nonlinear parabolic equations with varying degeneracy. The interested reader may consult \cite{Andrews-pjm} for motivation for the study of this flow.
For the short time existence, it was proved in \cite{Tso}  for $\alpha=1$,  and for any $\alpha>0$ in \cite{Chow} that the flow shrinks any smooth, uniformly convex body $M_0=\partial\Omega_0$ to a point $z_\infty$ in finite time $T>0$. An important differential Harnack estimate (also referred as Li-Yau-Hamilton type estimate) was later proved in \cite{Chow2} (see also \cite{Andrews-Harnack}). The current paper concerns the asymptotics of the solutions as the time approaches to the singular time $T$.

The study of the asymptotic behavior is equivalent to the large time behavior of the normalized flow, which is obtained by re-scaling about the final point to keep the enclosed volume fixed, and suitably re-parametrizing the time variable (see section \ref{sec:mono} for details):
\begin{equation}\label{gcf-alpha-nor}
\frac{\partial}{\partial t}X(x, t)= -\frac{K^{\alpha}(x, t)}{\aint_{\mathbb{S}^n}K^{\alpha-1}}\, \nu(x, t) +X(x,t).
\end{equation}
Here we write $\aint_{\mathbb{S}^n} f(x) d\theta(x) =\frac{1}{\omega_n}\int_{\mathbb{S}^n} f(x)\, d\theta(x)$ for any continuous function $f$ on $\bfS^n$, where $\omega_n=|\bfS^n|$, and we interpret $K$ as a function on $\bfS^n$ via the Gauss map diffeomorphism $\nu:\ M_t\to\bfS^n$.  It can be easily checked that $M_t=X(M,t)$ encloses a convex body $\Omega_t$ whose volume $|\Omega_t|$ changes according to the equation:
\begin{align*}
\frac{d}{dt}|\Omega_t|&=-\frac{1}{\aint_{\mathbb{S}^n}K^{\alpha-1}}\int_{M_t} K^\alpha +\int_{M_t}\langle X, \nu\rangle\\
&= -\omega_n +(n+1)|\Omega_t|.
\end{align*}
Hence if $|\Omega_0|=|B(1)|=\frac{\omega_n}{n+1}$, where $B(1)\subset \mathbb{R}^{n+1}$ is the unit ball, then $|\Omega_t|=|B(1)|$ for all $t$.

We briefly summarize previous work on the asymptotic behaviour of these flows:  Chow \cite{Chow} analyzed the case $\alpha=\frac1n$ and proved that solutions of the normalized flow converge to the unit sphere as $t\to\infty$,  by using pointwise estimates on the second fundamental form as previously carried out by Huisken \cite{Huisken-convex} for the mean curvature flow.  Convergence to spheres is known in some other special cases:  This was proved for $n=1$ and $\alpha> 1$ in \cite{A-COCU}, for $n=1$ and $\frac13<\alpha<1$ in \cite{CLSICF} and for $n=2$, $\alpha=1$ in \cite{An3} by the first author (see also \cite{Chou-Zhu}*{Proposition 2.3} for the case $n=1$), and for $n=2$ and $\frac12<\alpha<1$ by Chen and the first author \cite{A-cxz}.   The convergence of the flow was also completely analyzed for $\alpha=\frac{1}{n+2}$ by the first author in \cite{Andrews-jdg} for any dimension (the case $n=1$ was treated by Sapiro and Tannenbaum \cite{ST3}):  In this case the flow has a remarkable affine invariance, and the rescaled solutions converge to ellipsoids.  It seems a plausible conjecture (generalising a conjecture of Firey \cite{Firey} for the case $\alpha=1$) that solutions of \eqref{gcf-alpha-nor} should converge to spheres for any $\alpha>\frac{1}{n+2}$, but this is at present still open except for the cases mentioned above and the case $\alpha\ge1$ with central symmetry treated in this paper.

The sphere is a stationary solution of \eqref{gcf-alpha-nor} for any $\alpha$, corresponding to a `soliton' solution of \eqref{eq-gcf-alpha} which shrinks without change of shape.
Convergence to (possibly non-spherical) solitons was established for $\alpha\in (\frac{1}{n+2}, \frac{1}{n})$ in \cite{Andrews-pjm} by the first author, not only for the flow \eqref{eq-gcf-alpha}, but also for a family of anisotropic generalisations.  When $\alpha \in (0, \frac{1}{n+2})$, convergence to solitons was proved under the additional assumption that the isoperimetric ratio remains bounded, and examples were provided of non-spherical solitons for small $\alpha$.  However it was proved in \cite{A-instab} for $n=1$ and $0<\alpha<\frac13$ that the isoperimetric ratio of solutions generically becomes unbounded as the curve shrinks to a point, so the solutions of the normalized flow do not converge.  This is expected to remain true in higher dimensions for $\alpha<\frac{1}{n+2}$.
The methods of \cite{Andrews-pjm} do not apply for $\alpha>\frac1n$, and indeed for any such $\alpha$ there are examples of flows with smooth, strictly positive anisotropy where no positive lower bound on the Gauss curvature can hold.  This demonstrates that the analysis in these cases is much more subtle than for smaller $\alpha$.   Nevertheless, smooth convergence to solitons was established recently for the case $\alpha=1$ (without anisotropy) by the second and the third authors in \cite{Guan-N}.

The present paper generalizes the methods of \cite{Guan-N} to the more general case $\alpha>\frac1{n+2}$:
We prove smooth convergence of solutions of \eqref{gcf-alpha-nor} to solitons for arbitrary smooth, uniformly convex initial hypersurfaces, for any $\alpha>\frac{1}{n+2}$.  The crucial observation (in Lemma \ref{e-p}) is that an associated `entropy point', generalizing the classical Santal\'o point, lies strictly in the interior for any convex body of full dimension.  From this we arrive at a uniform lower bound on the support function (Theorem \ref{C0}), and this in turn implies a uniform lower bound on the Gauss curvature (Theorem \ref{thm-lower-k}). A new feature here is that we prove this without appealing to the different Harnack estimate of \cite{Andrews-Harnack} which extends to the anisotropy flow.  These estimates make it possible to use the methods of \cite{Andrews-pjm} to deduce a uniform $C^2$ estimate (Theorem \ref{c2-sharp}), and to conclude that the solution of the normalized flow for any smooth initial convex body $\Omega_0$ converges smoothly as $t\to\infty$ to a  uniformly convex soliton (Theorem \ref{c-infty}).

As a corollary of the main result and a soliton classification result (Proposition \ref{soliton}) we prove the smooth convergence to a round sphere for $\alpha\ge 1$, provided that the initial data is centrally symmetric, by adapting an argument of Firey \cite{Firey} for the case $\alpha=1$ to prove that centrally symmetric solitons are spheres.

Since the main convergence result is already known for the case $\alpha \in [\frac{1}{n+2}, \frac{1}{n}]$ we focus mainly on the case $\alpha \in (\frac{1}{n}, \infty)$,  but our techniques do provide a uniform treatment for all $\alpha \in (\frac{1}{n+2}, \infty)$.
 In the last section we also prove some stability estimates involving the entropy quantities. This generalizes the entropy nonnegativity result established in the next section. We expect applications of such estimates in the study of  convex bodies and their flows.

\section{The entropy and its basic properties}

Our argument is based on the analysis of an `entropy' functional, defined $\alpha \in (0, \infty)$ by
\begin{equation}\label{eq:def-e}
\mathcal{E}_{\alpha} (\Omega) :=
\sup_{z_0\in\Omega}{\mathcal E}_\alpha(\Omega,z_0),
\end{equation}
where
\begin{equation}\label{eq:def-e-z}
{\mathcal E}_\alpha(\Omega,z_0) :=
\begin{cases}
\frac{\alpha}{\alpha-1}\log\left(\aint_{\bfS^n} u_{z_0}(x)^{1-\frac1\alpha}\,d\theta(x)\right),&\alpha\neq 1;\\
\aint_{\bfS^n}\log u_{z_0}(x)\,d\theta(x),&\alpha=1.
\end{cases}
\end{equation}
where $u_{z_0}(x):=\sup_{z\in\Omega}\left\langle z-z_0,x\right\rangle$ is the \emph{support function} of $\Omega$ in direction $x$ with respect to $z_0$.  When $\alpha=1$ this agrees with the entropy used in \cite{Guan-N}, first introduced by Firey \cite{Firey}.
When $\alpha=\frac{1}{n+2}$, the entropy is related to the minimum volume of the polar dual body $\Omega^*_{z_0}$ of $\Omega$, which is attained at the Santal\'o point $z_s$ \cite{Sch}*{\S 10.5}. The general case was also used in \cite{Ivaki} recently.  We briefly recall  the definition of the polar dual:  Given $\Omega$ and $z_0\in \operatorname{Int}(\Omega)$, the polar dual of $\Omega$ with respect to $z_0$ is defined by
$$
\Omega^*_{z_0} -z_0=\{ w\, |\, \langle w, z-z_0\rangle \le 1, \forall\,  z\in \Omega\}.
$$
Writing $w$ in polar coordinates we have that
\begin{equation}\label{def-dual-body1}
\Omega^*_{z_0} -z_0=\{(r, x)\in (0, \infty)\times \mathbb{S}^n \, |\,  r u_{z_0}(x)\le 1\}.
\end{equation}
This implies the formula for the volume of the dual body \cite{Sch}*{\S 1.7}, \cite{MP}:
$$
|\Omega^*_{z_0}|=\int_0^{1/u_{z_0}(x)} \int_{\mathbb{S}^n}r^n \, d\theta\, dr=\frac{1}{n+1}\int_{\mathbb{S}^n}\frac{1}{u_{z_0}^{n+1}(x)}\, d\theta(x).
$$
The volume of  $\Omega^*_{z_0}$ varies with $z_0$, and is minimized at a unique point $z_s$ called the \emph{Santal\'o point}. We also denote $\Omega^*_{z_s}$ by $\Omega^*_s$.  The Blaschke-Santal\'o inequality (\cite{Blaschke}*{p.208}, \cite{Santalo}, \cite{MP}) states
\begin{eqnarray}
|\Omega|\cdot |\Omega_s^*|\le |B(1)|^2. \label{BS}
\end{eqnarray}
If $|\Omega|=|B(1)|$, then this
implies  $|\Omega^*|\le |B(1)|=\frac{\omega_n}{n+1}$.

\begin{proposition}\label{E-incr} For any fixed convex body of full dimension and $z_0\in\operatorname{Int}(\Omega)$, the entropy ${\mathcal E}_{\alpha}(\Omega,z_0)$ is continuous and nondecreasing in $\alpha$ (strictly unless $u_{z_0}$ is constant).
\end{proposition}

\begin{proof}
For continuity at $\alpha=1$, see for example \cite{GT}*{Problem 7.1}.  Monotonicity for either $\alpha>1$ or $0<\alpha<1$ is a direct consequence of the H\"older inequality (see \cite{GT}*{page 146}).
\end{proof}

\begin{corollary}\label{key} Let $\Omega$ be a bounded convex body in $\mathbb R^{n+1}$ with $|\Omega|=|B(1)|$. Then for each $\alpha>\frac{1}{n+2}$ we have ${\mathcal E}_\alpha(\Omega)\geq 0$, with equality if and only if $\Omega$ is a ball.
\end{corollary}


\begin{proof}
We observe that
$$
{\mathcal E}_{\frac1{n+2}}(\Omega,z_0) = -\frac{1}{n+1}\log\left(\aint_{\bfS^n}u_{z_0}(x)^{-(n+1)}\,d\theta(x)\right) = -\frac{1}{n+1}\log\left(\frac{n+1}{\omega_n}\left|\Omega^*_{z_0}\right|\right).
$$
It follows that the supremum ${\mathcal E}_{\frac1{n+2}}(\Omega)$ is attained when $z_0$ is the Santal\'o point $z_s$, and that ${\mathcal E}_{\frac1{n+2}}(\Omega)\geq 0$ by \eqref{BS}.  The Corollary follows by the monotonicity in $\alpha$ of Proposition \ref{E-incr}:  We have ${\mathcal E}_\alpha(\Omega)\geq {\mathcal E}_\alpha(\Omega,z_s)\geq {\mathcal E}_{\frac1{n+2}}(\Omega,z_s)\geq 0$.
\end{proof}


 Before proving the key estimates involving the entropy, we provide a geometric interpretation in terms of a weighted volume of the dual body,
analogous to \cite{Guan-N}*{Proposition 2.2}:

\begin{proposition}\label{prop:geom-entropy}  Let $\Omega^{0}_{z_0}=\cap_{z\in\Omega}\{x\in\bfR^{n+1}:\ \left\langle z-z_0,x\right\rangle\leq 1\} =\Omega^*_{z_0}-z_0$.
Then for $\alpha <1$,
$$
\omega_n\E^{\frac{\alpha-1}{\alpha}\mathcal{E}_\alpha(\Omega, z_0)}=\int_{\mathbb{S}^n} u_{z_0}^{1-\frac{1}{\alpha}}(x)\, d\theta(x)=\frac{1}{\frac{1}{\alpha}-1}\int_{\Omega^0_{z_0}}|w|^{\frac{1}{\alpha}-2-n}\, dw
$$
and for $\alpha>1$
$$
\omega_n\E^{\frac{\alpha-1}{\alpha}\mathcal{E}_\alpha(\Omega, z_0)}=\int_{\mathbb{S}^n} u_{z_0}^{1-\frac{1}{\alpha}}(x)\, d\theta(x)=\frac{1}{1-\frac{1}{\alpha}}\int_{\mathbb{R}^{n+1}\setminus \Omega^0_{z_0}} |w|^{\frac{1}{\alpha}-2-n}\, dw.
$$
That is, $\int_{\mathbb{S}^n}  u^{1-\frac{1}{\alpha}}_{z_0}$ is the weighted  volume of $\Omega^0_{z_0}$ for $\alpha<1$,  or the weighted volume of $\mathbb{R}^{n+1}\setminus \Omega^0_{z_0}$ for $\alpha>1$, with respect to the measure $|w|^{\frac1\alpha-n-2}\,dw$.  In particular, for any $z_0$ with  $|\Omega^*_{z_0}|\le |B(1)|$, we have
$\int_{\mathbb{S}^n}  u_{z_0}^{1-\frac{1}{\alpha}}(x)\, d\theta(x)\le 1$ for $\alpha<1$, and $\int_{\mathbb{S}^n}  u_{z_0}^{1-\frac{1}{\alpha}}(x)\, d\theta(x)\ge 1$ for $\alpha>1$. Moreover, for $0<\alpha<1$, there is a unique point $z_0\in \operatorname{Int}(\Omega)$ such that $\mathcal{E}_\alpha(\Omega)=\mathcal{E}_\alpha(\Omega, z_0)$, which satisfies
$$
\int_{\Omega^0_{z_0}} \frac{w}{|w|^{n+2-\frac{1}{\alpha}}}\, dw=0.
$$
That is, $z_0$ is the unique point for which the center of mass of $\Omega^0_{z_0}$ with respect to the weighted  measure $\frac{d\, w}{|w|^{n+2-\frac{1}{\alpha}}}$ lies at the origin.
\end{proposition}
\begin{proof} Direct calculation yields
\begin{eqnarray*}
\aint_{\mathbb{S}^n}  u_{z_0}^{1-\frac{1}{\alpha}}(x)\, d\theta(x)&=&\frac{1}{\frac{1}{\alpha}-1}\aint_{\mathbb{S}^n}\int_0^{\frac{1}{u_{z_0}(x)}}r^{\frac{1}{\alpha}-2}
\, dr\, d\theta(x)\\
&=&\frac{1}{\frac{1}{\alpha}-1} \aint_{\mathbb{S}^n} \int_0^{\frac{1}{u}} r^{\frac{1}{\alpha}-2-n}d\, w   \\
&=& \frac{1}{\frac{1}{\alpha}-1}\cdot \frac{1}{\omega_n} \int_{\Omega^{0}_{z_0}} |w|^{\frac{1}{\alpha}-n-2}\, d\, w.
\end{eqnarray*}
 This proves the identity for $\alpha<1$. For $\alpha>1$ the computation is the same.

 If $|\Omega^*_{z_0}|\le |B(1)|$, then $|B(1)\setminus \Omega^0_{z_0}|\ge |\Omega^0_{z_0}\setminus B(1)|$. Hence we have
 \begin{eqnarray*}
 \frac{1}{\frac{1}{\alpha}-1} \int_{\Omega^{0}_{z_0}} |w|^{\frac{1}{\alpha}-n-2}\, d\, w &=&   \frac{1}{\frac{1}{\alpha}-1} \left(\int_{\Omega^{0}_{z_0}\cap B(1)}+\int_{\Omega^0_{z_0}\setminus B(1)} |w|^{\frac{1}{\alpha}-n-2}\, d\, w\right)\\
 &\le&  \frac{1}{\frac{1}{\alpha}-1} \left(\int_{\Omega^{0}_{z_0}\cap B(1)}+\int_{B(1)\setminus \Omega^0_{z_0}}  |w|^{\frac{1}{\alpha}-n-2}\, d\, w\right)\\
 &=& \omega_n.
 \end{eqnarray*}
For the inequality above we used that $|w|^{\frac{1}{\alpha}-n-2}\le 1$ for $w\in \Omega^0_{z_0}\setminus B(1)$ while $|w|^{\frac{1}{\alpha}-n-2}\ge 1$ for $w\in B(1)\setminus \Omega^0_{z_0}$. The last equality is via a simple calculation.

For $\alpha>1$ the proof can be done by reversing some of the estimates above. The last statement on the center of mass  follows by a calculation similar to the proof of the first identity in the proposition.\end{proof}

\begin{remark} From the proof one can derive the following identities from which Proposition \ref{key} is also obvious. For $\frac{1}{n+2}\le \alpha< 1$,
\begin{eqnarray}\label{entropy-formula1}
 \E^{\frac{\alpha-1}{\alpha}\mathcal{E}(\Omega, z_0)}&=& 1 -\frac{1}{(\frac{1}{\alpha}-1)\omega_n} \left(d_{n+2-\frac{1}{\alpha}} (\Omega^0_{z_0}, B(1))+(|B(1)\setminus \Omega^0_{z_0}|-|\Omega^0_{z_0}\setminus B(1)|)\right),
\end{eqnarray}
where $d_{n+2-\frac{1}{\alpha}}(A, B)$ denotes the measure of the symmetric difference of two sets $A, B$ with respect to  the measure $\left||w|^{\frac{1}{\alpha}-n-2}-1\right| dw$.
For $\alpha>1$,
\begin{eqnarray}\label{entropy-formula2}
 \E^{\frac{\alpha-1}{\alpha}\mathcal{E}(\Omega, z_0)}&=& 1 +\frac{1}{(1-\frac{1}{\alpha})\omega_n} \left(d_{n+2-\frac{1}{\alpha}} (\Omega^0_{z_0}, B(1))+(|B(1)\setminus \Omega^0_{z_0}|-|\Omega^0_{z_0}\setminus B(1)|)\right).
\end{eqnarray}

\end{remark}

The next lemma is important in obtaining the crucial estimates for the flow.

\begin{lemma}\label{e-p} If $\Omega$ is a bounded convex domain with  $\operatorname{Int}(\Omega)\ne \emptyset$, then there exists a unique point $z_e\in \operatorname{Int}(\Omega)$ such that $\mathcal{E}_\alpha (\Omega)=\mathcal{E}_\alpha(\Omega,z_e)$. Moreover
\begin{equation}\label{1stvar}
\int_{\mathbb S^n}\frac{x_j}{u^{\frac{1}{\alpha}}_{z_e}(x)}\, d\theta(x)=0.
\end{equation}
Furthermore, if $z\neq z_e$ is in $\operatorname{Int}(\Omega)$ then $\mathcal{E}_\alpha(\Omega,z)<\mathcal{E}_\alpha (\Omega)$.
\end{lemma}
\begin{proof} Apply the argument in the proof  of Lemma 2.3 in \cite{Guan-N}.
\end{proof}

\begin{remark}
It is crucial to our later argument that the `entropy point' $z_e$ constructed in Lemma \ref{e-p}  is in the interior of $\Omega$.  The proof of this fact is the only place where our argument fails in the more general situation of the anisotropic Gauss curvature flows considered in \cite{Andrews-pjm}.  Accordingly, our main result of smooth convergence to solitons holds for any anisotropy for which it can be etablished that the entropy point is in the interior of the domain.
\end{remark}
\medskip

The next result allows us to control the geometry of a convex body $\Omega$ in terms of the entropy $\mathcal{E}_\alpha (\Omega)$.
  Let $\rho_{+}(\Omega) $  ($\rho_{-}(\Omega)$) be the outer (inner) radius of  a convex body $\Omega$. By definition, the outer radius is the  radius of the smallest ball which contains $\Omega$ and the inner radius is the radius of the biggest ball which is enclosed by $\Omega$. There is also a width function $w(x)$ which is defined as $u_{z_0}(x)+u_{z_0}(-x)$, where $u_{z_0}$ is the support function with respect to $z_0$.
 It is clear that these are independent of the choice of $z_0$. Let $w_{+}$ and $w_{-}$ denote the maximum and minimum of $w(x)$. Recall from  \cites{Jung, Steinhagen} that
 \begin{equation}\label{andrews-eq1}
\frac12 w_+\le \rho_{+}\le w_{+}\sqrt{\frac{2(n+1)}{n+2}}, \quad \quad \frac12 w_-\ge\rho_{-}\ge \frac{w_{-}}{2\sqrt{n+1}}.\end{equation}

\begin{proposition}\label{rho-e}
\begin{enumerate}[label={(\roman*)}]
\item For each
$\alpha>\frac{1}{n+2}$ there exist positive constants $\beta$ and $C$ depending only on $\alpha$ and $n$ such that for every convex body $\Omega$ with $|\Omega|=|B(1)|$,
\begin{equation}\label{C0-lower}
\min\{\rho_{-}(\Omega),  w_{-}(\Omega)\} \ge C^{-1}\E^{-\beta\mathcal{E}_\alpha(\Omega)},
\end{equation}
and
\begin{equation}\label{C0-upper}
\max\{w_{+}(\Omega), \rho_{+}(\Omega)\}\le C \E^{n\beta\mathcal{E}_\alpha(\Omega)}.\end{equation}
\item For any $\alpha\in(0,\frac1{n+2})$ there exists a positive constant $C$ depending only on $\alpha$ and $n$ such that for any convex body $\Omega$ of full dimension,
\begin{equation}\label{C0-l-upper}
\left|{\mathcal E}_\alpha(\Omega)-\frac{1-\alpha(n+2)}{1-\alpha}\log\left(\rho_-(\Omega)\right)\right|\leq C.
\end{equation}
\end{enumerate}
\end{proposition}
\begin{proof}
We use the result of John and L\"owner \cite{John} which provides for any convex body $\Omega$ of full dimension the existence of an ellipsoid $E$ and centre $z_0$ such that $E+z_0\subset\Omega\subset (n+1)E+z_0$.  In our case this means that the volume of $E$ is comparable to that of the unit ball.  The inclusion also implies that $\frac{1}{n+1}E^*\subset \Omega^0_{z_0}\subset E^*$.  The principal axis theorem allows us to rotate so that $E = \{z:\ \sum_{i=1}^{n+1}a_i^{2}z_i^2\leq 1\}$, so that the principal semi-axes of $E$ have lengths $a_i^{-1}$, $1=1,\dots,n+1$.  Then $E^* = \{ w:\ \sum_{i=1}^{n+1}a_i^{-2}w_i^2\leq 1\}$ has principal semi-axes of length $a_i$, $i=1,\dots,n+1$, which we arrange in non-decreasing order.  We estimate the entropy using weighted area expression in Proposition \ref{prop:geom-entropy}, observing that $|w|\geq |w_{n+1}|$ and that $E^*\subset Q=\{w:\ |w_i|\leq a_i,\ i=1,\dots,n+1\}$.  Since ${\mathcal E}_\alpha(\Omega)$ is nondecreasing in $\alpha$, the result for larger $\alpha$ follows from that for smaller $\alpha$, so in the following we assume that $\frac{1}{n+2}<\alpha<\frac{1}{n+1}$.
This gives
\begin{align*}
\frac{1-\alpha}{\alpha}\omega_n\E^{-\frac{1-\alpha}{\alpha}{\mathcal E}_\alpha(\Omega)}&\leq \frac{1-\alpha}{\alpha}\omega_n\E^{-\frac{1-\alpha}{\alpha}{\mathcal E}_\alpha(\Omega,z_0)}\\
&= \int_{\Omega^0_{z_0}}|w|^{\frac1\alpha-2-n}\,d{\mathcal H}^{n+1}\\
&\leq \int_{E^*}|w_{n+1}|^{\frac1\alpha-2-n}\,d{\mathcal H}^{n+1}\\
&\leq \int_{Q}|w_{n+1}|^{\frac1\alpha-2-n}\,d{\mathcal H}^{n+1}\\
& = \int_{-a_{n+1}}^{a_{n+1}} 2^n\left(\prod_{i=1}^na_i\right)|z|^{\frac1\alpha-2-n}\,dz\\
&=\frac{2^{n+1}}{\frac1\alpha-n-1}\left(\prod_{i=1}^na_i\right)a_{n+1}^{\frac1\alpha-n-1}\\
&=\frac{2^{n+1}}{\frac1\alpha-n-1}\left(\prod_{i=1}^{n+1}a_i\right)a_{n+1}^{\frac1\alpha-n-2}.
\end{align*}
Now we observe that $2^{n+1}\prod_{i=1}^{n+1}a_i$ is the volume of $Q$, which is $(n+1)2^{n+1}\omega_n^{-1}$ times the volume of $E^*$. Since $E$ is an ellipsoid we have $|E|\cdot|E^*|=\frac{\omega_n^2}{(n+1)^2}$, and we also have $|\Omega|=\frac{\omega_n}{n+1}\leq (n+1)^{n+1}|E|$ by the inclusion $\Omega\subset (n+1)E+z_0$.  This gives $\prod_{i=1}^{n+1}a_i\leq \left((n+1)\right)^{n+1}$, and we conclude that
$$
\frac{1-\alpha}{\alpha}\omega_n\E^{-\frac{1-\alpha}{\alpha}{\mathcal E}_\alpha(\Omega)}\leq \frac{\left(2(n+1)\right)^{n+1}}{\frac1\alpha-n-1}a_{n+1}^{\frac1\alpha-n-2}.
$$
Finally, we observe that $a_{n+1}^{-1}=\frac{w_-(E)}{2}$, so by the inclusion $E+z_0\subset K$ we have $a_{n+1}^{-1}\leq \frac{w_-(\Omega)}{2}$, implying the estimate
$$
\E^{\frac{1-\alpha}{\alpha}{\mathcal E}_\alpha(\Omega)}\geq\frac{(1-\alpha)(\frac1\alpha-n-1)\omega_n}{\alpha(n+1)^{n+1}}\left(\frac{w_-(\Omega)}{2}\right)^{-(n+2-\frac1\alpha)},
$$
and we have proved the estimate \eqref{C0-lower} (since $w_-$ and $\rho_-$ are comparable in view of \eqref{andrews-eq1}).  To obtain the estimate \eqref{C0-upper}, we observe that $\Omega$ contains the convex hull of the union of a diameter of $\Omega$ with an insphere of $\Omega$, and hence has volume no less than $\frac{\omega_{n-1}}{n(n+1)}\rho_+\rho_-^n$.  Since $|\Omega|=\frac{1}{n+1}\omega_n$, it follows that $\rho_+\leq \frac{n\omega_n}{\omega_{n-1}}\rho_-^{-n}$, so that \eqref{C0-upper} follows from \eqref{C0-lower}.

The estimate (\ref{C0-l-upper}) follows by a similar argument in the case $0<\alpha<\frac{1}{n+2}$.
\end{proof}

We remark that the estimate \eqref{C0-l-upper} will play no further role in our argument (since our results concern only the case $\alpha>\frac{1}{n+2}$), but has an interesting consequence for the flow \eqref{gcf-alpha-nor} with $\alpha<\frac{1}{n+2}$:  By the second inequality in \eqref{C0-l-upper}, any initial convex body $\Omega_0$ with $|\Omega_0|=|B(1)|$ with large diameter (equivalently, small inradius) has entropy far below zero (whereas $B(1)$ has entropy equal to zero).  The monotonicity results proved in the next section for the entropy under \eqref{gcf-alpha-nor} imply that the entropy remains far below zero, and the second inequality in \eqref{C0-l-upper} then implies that the diameter remains large.  In particular the solution of \eqref{gcf-alpha-nor} from such initial data remains far from spherical.  It seems a plausible conjecture that in such situations the diameter will always become unbounded under \eqref{gcf-alpha-nor} if it is initially sufficiently large.

\section{Monotonicity formulae and geometric bounds for the solutions}\label{sec:mono}

It will be useful during our argument to use both the un-normalised flow \eqref{eq-gcf-alpha} and the normalised flow \eqref{gcf-alpha-nor}, so we provide a more precise account of the relation between the two here:  Given a solution $\tilde X:\ M\times[0,T)\to\bfR^{n+1}$ of \eqref{eq-gcf-alpha}, we translate so that $\tilde X_\tau$ shrinks to the origin as $\tau$ approaches $T$, and make the following definitions:
\begin{equation}\label{eq:newtime}
t: =  \frac{1}{n+1}\log\left(\frac{|B(1)|}{|\tilde\Omega_\tau|}\right),
\end{equation}
and (writing $\tilde\Omega_\tau$ for the region enclosed by $\tilde M_t=\tilde X(M,\tau)$)
\begin{equation}\label{eq:rescale}
X(x,t):= \left(\frac{|B(1)|}{|\tilde\Omega_\tau|}\right)^{\frac1{n+1}}\tilde X(x,\tau)= \E^t\tilde X(x,\tau).
\end{equation}
A direct computation then shows that $X$ is a solution of \eqref{gcf-alpha-nor}.  Note that the results of Tso \cite{Tso} and Chow \cite{Chow} guarantee that $|\tilde\Omega_\tau|$ approaches zero as $\tau$ approaches $T$, and consequently $t$ approaches infinity as $\tau$ approaches $T$ and the solution $X_t$ exists for all positive times.  We remark that the entropy of $X$ is related to that of $\tilde X$ as follows:
\begin{equation}\label{eq:rescale-entropy}
{\mathcal E}_\alpha(\Omega_t,e^{t}z_0) = {\mathcal E}_\alpha(\tilde\Omega_\tau, z_0)-\frac1{n+1}\log\left(\frac{|\tilde\Omega_\tau|}{|B(1)|}\right).
\end{equation}
The monotonicity of the entropy $\mathcal{E}_\alpha(\Omega_t)$ along the flow \eqref{eq-gcf-alpha}--\eqref{gcf-alpha-nor} was first proved in \cite{Andrews-IMRN}. The result below is a refinement of it.

\begin{theorem}\label{mono1} Assume that $\alpha>0$.  Then
\begin{enumerate}[label={(\roman*).}]
\item Under the un-normalised flow \eqref{eq-gcf-alpha}, if $z_0\in\operatorname{Int}(\tilde\Omega_{\tau_1})$ then $z_0\in\operatorname{Int}(\tilde\Omega_\tau)$ for all $\tau\in[0,\tau_1]$, and ${\mathcal E}(\tilde\Omega_\tau,z_0)-\frac1{n+1}\log\left(\frac{|\tilde\Omega_\tau|}{|B(1)|}\right)$ is non-increasing, strictly unless $\tilde\Omega_\tau$ is a soliton shrinking to $z_0$.  Furthermore we have the following for $0\leq \tau_0<\tau_1<T$:
\begin{equation}\label{eq:entropy-decrease}
{\mathcal E}_\alpha(\tilde\Omega_{\tau_1},z_0)-{\mathcal E}_\alpha(\tilde\Omega_{\tau_0},z_0)-\frac{1}{n+1}\log\left(\frac{|\tilde\Omega_{\tau_1}|}{|\tilde\Omega_{\tau_0}|}\right)=-\int_{\tau_0}^{\tau_1}
\left[\frac{\int_{\bfS^n}\tilde f_{z_0}^{1+\frac1\alpha}\,d\tilde\sigma}{\int_{\bfS^n}\tilde f_{z_0}^{\frac1\alpha}\,d\tilde\sigma}-\frac{\int_{\bfS^n}\tilde f_{z_0}\,d\tilde\sigma}{\int_{\bfS^n}\,d\tilde\sigma}
\right]\,d\tau\leq 0,
\end{equation}
where $\tilde f_{z_0}(x,\tau)=\frac{\tilde K^\alpha(x,\tau)}{\tilde u_{z_0}(x, \tau)}$, $d\tilde\sigma(x)=\frac{\tilde u_{z_0}(x,\tau)}{\tilde K(x, \tau)}\, d\theta(x)$.
\item Under the normalised flow \eqref{gcf-alpha-nor}, ${\mathcal E}_\alpha(\Omega_t)$ is non-increasing, strictly unless $\Omega_0$ is a soliton.  Furthermore,  $\mathcal{E}_\alpha^\infty:=\lim_{t\to \infty} \mathcal{E}_{\alpha}(\Omega_t)$ exists if $\alpha\geq\frac1{n+2}$, and we have the following inequality:
\begin{equation}\label{eq:entropy-nor}
{\mathcal E}_\alpha(\Omega_{t_0})-{\mathcal E}_\alpha^\infty\leq -\int_{t_0}^\infty \left[\frac{\int_{\bfS^n}f^{1+\frac1\alpha}\,d\sigma_t\cdot \int_{\bfS^n}\,d\sigma_t}
{\int_{\bfS^n}f^{\frac1\alpha}\,d\sigma_t\cdot\int_{\bfS^n} f\,d\sigma_t}-1\right]\,dt\leq 0.
\end{equation}
Here $ f(x,t)=\frac{ K^\alpha(x,t)}{ u(x, t)}$, $d\sigma_t(x)=\frac{ u(x,t)}{ K(x, t)}\, d\theta(x)$.
\end{enumerate}
\end{theorem}
\begin{proof} In terms of the support function the flow (\ref{eq-gcf-alpha}) can be written as
\begin{equation}\label{eq-gcf-s}
\tilde u_\tau(x, \tau)=-\tilde K^\alpha(x, \tau),
\end{equation}
while the normalised flow \eqref{gcf-alpha-nor} becomes the following:
\begin{equation}\label{eq-gcf-s-nor}
u_t(x,t)=-\frac{K^\alpha(x,t)}{\aint_{\mathbb{S}^n}  K^{\alpha-1}(y, \tau) d\theta(y)}+u(x,t).
\end{equation}
 Since the origin is assumed to be the shrinking limit of the un-normalized flow, we have $\tilde u(x,\tau)>0$ under \eqref{eq-gcf-alpha}, hence also $u(x,t)>0$ for all $(x,t)\in\mathbb{S}^n\times (0, +\infty)$ under \eqref{gcf-alpha-nor}.
 We first prove the monotonicity under the un-normalised flow:  Since $\tilde M_t$ is shrinking we have $\tilde\Omega_{\tau_1}\subset\tilde\Omega_{\tau}$ for any $\tau<\tau_1$, and if $z_0\in\operatorname{Int}(\tilde\Omega_{\tau_1})$ then we have $\tilde u_{z_0}(x,\tau)>0$ for all $(x,\tau)\in\bfS^n\times[0,\tau_1]$.
 We note that $\int_{\bfS^n} 1\, d\tilde\sigma_\tau = (n+1)|\tilde\Omega_{\tau}|$, $\int_{\bfS^n}\tilde f_{z_0}\,d\tilde\sigma = \int_{\bfS^n} \tilde K^{\alpha-1}\,d\theta$, and $\int_{\bfS^n}\tilde f_{z_0}^{\frac1\alpha}\,d\tilde\sigma = \int_{\bfS^n}\tilde u_{z_0}^{1-\frac1\alpha}\,d\theta$, while $\frac{\partial}{\partial\tau}\int_{\bfS^n}\tilde u_{z_0}^{1-1/\alpha}d\theta = \frac{\alpha-1}{\alpha}\int_{\bfS^n}\tilde u_{z_0}^{-\frac1\alpha}\tilde K^\alpha\,d\theta = \frac{\alpha-1}{\alpha}\int_{\bfS^n}\tilde f_{z_0}^{1+\frac1\alpha}\,d\tilde\sigma$.  This gives
$$
 \frac{\partial}{\partial\tau}\left({\mathcal E}_\alpha(\tilde\Omega_\tau)-\frac1{n+1}\log\left(\frac{|\tilde\Omega_{\tau}|}{|B(1)|}\right)\right)
=-\left[\frac{\int_{\bfS^n}\tilde f_{z_0}^{1+\frac1\alpha}\,d\tilde\sigma}{\int_{\bfS^n}\tilde f_{z_0}^{\frac1\alpha}\,d\tilde\sigma}-\frac{\int_{\bfS^n}\tilde f_{z_0}\,d\tilde\sigma}{\int_{\bfS^n}\,d\tilde\sigma}\right].
$$
The right-hand side is non-positive by the H\"older inequality, with equality if and only if $\tilde f_{z_0}$ is constant, in which case $\tilde K^{\alpha} = c\tilde u_{z_0}$ and $\tilde\Omega_{\tau}$ is a soliton shrinking to $z_0$.

 The monotonicity for the normalized flow follows:  If $t_1>t_0$ then ${\mathcal E}_\alpha(\Omega_{t_i}) = {\mathcal E}_\alpha(\tilde\Omega_{\tau_i})-\frac{1}{n+1}\log\left(\frac{|B(1)|}{|\tilde\Omega_{\tau_i}|}\right)$ for $i=1,2$, where $\tau_i$ corresponds to $t_i$ under \eqref{eq:newtime}.  Let $\tilde z_1\in\operatorname{Int}(\tilde\Omega_{\tau_1})$ be the entropy point of $\tilde\Omega_{\tau_1}$.  Then we have
 \begin{align}\label{eq:e-mono}
 {\mathcal E}_\alpha(\Omega_{t_1})&={\mathcal E}_\alpha(\tilde\Omega_{\tau_1})+\frac1{n+1}\log\left(\frac{|B(1)|}{|\tilde\Omega_{\tau_1}|}\right)\\
 &={\mathcal E}_\alpha(\tilde\Omega_{\tau_1},\tilde z_1)+\frac1{n+1}\log\left(\frac{|B(1)|}{|\tilde\Omega_{\tau_1}|}\right)\notag\\
 &={\mathcal E}_\alpha(\tilde\Omega_{\tau_0},\tilde z_1)+\frac1{n+1}\log\left(\frac{|B(1)|}{|\tilde\Omega_{\tau_0}|}\right)-\int_{\tau_0}^{\tau_1}\left[\frac{\int_{\bfS^n}\tilde f_{\tilde z_1}^{1+\frac1\alpha}\,d\tilde\sigma}{\int_{\bfS^n}\tilde f_{\tilde z_1}^{\frac1\alpha}\,d\tilde\sigma}-\frac{\int_{\bfS^n}\tilde f_{\tilde z_1}\,d\tilde\sigma}{\int_{\bfS^n}\,d\tilde\sigma}
\right]\,d\tau\notag\\
&\leq {\mathcal E}_\alpha(\tilde\Omega_{\tau_0})+\frac1{n+1}\log\left(\frac{|B(1)|}{|\tilde\Omega_{\tau_0}|}\right)-\int_{\tau_0}^{\tau_1}\left[\frac{\int_{\bfS^n}\tilde f_{\tilde z_1}^{1+\frac1\alpha}\,d\tilde\sigma}{\int_{\bfS^n}\tilde f_{\tilde z_1}^{\frac1\alpha}\,d\tilde\sigma}-\frac{\int_{\bfS^n}\tilde f_{\tilde z_1}\,d\tilde\sigma}{\int_{\bfS^n}\,d\tilde\sigma}
\right]\,d\tau\notag\\
&={\mathcal E}_\alpha(\Omega_{t_0})-\int_{\tau_0}^{\tau_1}\left[\frac{\int_{\bfS^n}\tilde f_{\tilde z_1}^{1+\frac1\alpha}\,d\tilde\sigma}{\int_{\bfS^n}\tilde f_{\tilde z_1}^{\frac1\alpha}\,d\tilde\sigma}-\frac{\int_{\bfS^n}\tilde f_{\tilde z_1}\,d\tilde\sigma}{\int_{\bfS^n}\,d\tilde\sigma}
\right]\,d\tau.\notag
 \end{align}
 The monotonicity of ${\mathcal E}_\alpha(\Omega_t)$ follows.  Since ${\mathcal E}_\alpha(\Omega_t)$ is bounded below by zero for $\alpha\geq \frac{1}{n+2}$ by Corollary \ref{key}, it follows that ${\mathcal E}_\alpha^\infty$ exists and is non-negative.  The expression \eqref{eq:entropy-nor} follows by taking the limit $t\to\infty$ and changing variables from $\tau$ to $t$ in the integral, noting that in this limit $\tau\to T$ the entropy point  $\tilde z_\tau\in\tilde\Omega_\tau$ converges to the origin.
\end{proof}

\begin{corollary}\label{non-collapsing} Let $M_t=\partial\Omega_t$ be a solution to (\ref{gcf-alpha-nor}) with $|\Omega_t|=|B(1)|$ and $\alpha>\frac{1}{n+2}$.  Then there exists $C=C(\Omega_0)$ such that
\begin{equation}\label{est-rho}
\max\{w_{+}(\Omega_t), \rho_{+}(\Omega_t)\}\le C, \quad \min\{\rho_{-}(\Omega_t),  w_{-}(\Omega_t)\} \ge \frac{1}{C},
\end{equation}
for all $t>0$.
\end{corollary}

\begin{proof}
Since the entropy is non-increasing, the result follows by Proposition \ref{rho-e}.
\end{proof}

The following consequence gives some geometric meaning to the limiting point in terms of the entropy $\mathcal{E}_\alpha (\Omega, z)$.
\begin{corollary}\label{limit-meaning} Let $\Omega$ be a smooth closed bounded  convex domain with $|\Omega|=|B(1)|$. Assume that origin is the shrinking limit of the flow (\ref{eq-gcf-alpha}) with support function $u$. Then
$$
\aint_{\mathbb{S}^n} u^{1-\frac{1}{\alpha}}(x)\, d\theta(x) \le 1, \mbox{ if } \frac{1}{n+2}\le \alpha <1;\quad  \aint_{\mathbb{S}^n} u^{1-\frac{1}{\alpha}}(x)\, d\theta(x)\ge 1 \mbox{ if } \alpha >1.
$$
\end{corollary}
\begin{proof} This follows from the expression in the third line of \eqref{eq:e-mono}, after taking the limit $\tau_1\to T$.
\end{proof}

\section{$C^0$--Estimates for the flow by $K^\alpha$}

The main result of this section is to establish a uniform lower bound  for the support function $u(x, t)$ under \eqref{gcf-alpha-nor} for any $\alpha>\frac{1}{n+2}$.  Note that since the entropy is non-increasing under the flow, the estimates of Proposition \ref{rho-e} provide a lower bound for the inradius or minimum width, so we can easily obtain a lower bound on the support function $u$ if we are willing to translate the solution.  The subtle point in the following theorem is to obtain a lower bound for the support function about the origin (which is chosen to be the shrinking limit of the solution) without translations.

\begin{theorem}\label{C0}
Suppose $u(x, t)>0$ is the solution of (\ref{eq-gcf-s-nor}) with initial data $u(x, 0)=u_0(x)$, where $u_0(x)$ is the support function of a convex body $\Omega_0$ with $|\Omega_0|=|B(1)|$, and the corresponding solution of \eqref{eq-gcf-alpha} converges to the origin.  Then there exists $\epsilon =\epsilon(n,\mathcal{E}(\Omega_0))>0$ and $T_0=T(\Omega_0)$ such that for $t\ge T_0$ and $x\in\bfS^n$,
\begin{equation}\label{eq-c01} u(x,t)\ge \epsilon.\end{equation}
 \end{theorem}

The proof is built upon three elementary lemmas.  In order to state these, we first define for each $\rho\in(0,1)$ the following collection of convex bodies:
\begin{equation}
\Gamma_\rho=\left\{ \Omega \subset \mathbb R^{n+1} \text{\ compact, convex\ }\Big|\   \{\rho_+(\Omega), \rho_{-}(\Omega)\}\subset \left[\rho, \frac{1}{\rho}\right] \right\}.
\end{equation}
For each compact convex body $\Omega$ of full dimension, we denote by $z_e(\Omega)$ the `entropy point' of $\Omega$ characterised by Lemma \ref{e-p}.

\begin{lemma}\label{e-stable}  $\mathcal{E}_\alpha(\Omega,z)$ is a concave function of $z$ for any $\Omega$.  Furthermore, for each $\alpha>0$ and $\rho\in(0,1)$ there exists $D>0$ such that for every $\Omega\in\Gamma_\rho$ and every $z\in\operatorname{Int}(\Omega)$,
$$
\mathcal{E}_\alpha(\Omega, z)\leq \mathcal{E}_\alpha(\Omega)-\min\{1,D|z-z_e(\Omega)|^2\}.
$$
\end{lemma}

\begin{proof} At $z_e$, $\mathcal{E}_\alpha(\Omega, z)$ attains its maximum with respect to $z$.  Fix $z\neq z_e$, and let $z(s)=z_e+s\vec{a}$, where $\vec{a}=\frac{z-z_e}{|z-z_e|}$.  Define
$F(s)=\mathcal{E}_{\alpha}(\Omega, z(s))$.
By assumption we have $F(0)=\mathcal{E}_\alpha(\Omega)$ and $F'(0)=0$. Direct calculation shows that
\begin{align}\label{eq:f''ineq}
F''(s)&=-\frac{\aint_{\bfS^n}u_{z(s)}(x)^{-\frac{\alpha+1}\alpha}x\cdot\vec{a}\,d\theta(x)}{\alpha\aint_{\bfS^n}u_{z(s)}(x)^{\frac{\alpha-1}\alpha}\,d\theta(x)}+\left(\frac1\alpha-1\right)\frac{\left(\aint_{\bfS^n}u_{z(s)}(x)^{-\frac1\alpha}x\cdot\vec{a}\,d\theta(x)\right)^2}{\left(\aint_{\bfS^n}u_{z(s)}(x)^{\frac{\alpha-1}\alpha}\,d\theta(x)\right)^2}\notag\\
&\leq-\frac1{\alpha\left(\aint_{\bfS^n}u_{z(s)}^{\frac{\alpha-1}\alpha}\,d\theta\right)^{2}}
\left\{\aint_{\bfS^n}u_{z(s)}^{\frac{\alpha-1}\alpha}\,d\theta\cdot\aint_{\bfS^n} u_{z(s)}^{-\frac{1+\alpha}\alpha} \left(x\cdot\vec{a}\right)^2\,d\theta-\left(\aint_{\bfS^n}u^{-\frac1\alpha} x\cdot\vec{a}\,d\theta\right)^2
\right\}
\end{align}
The bracket on the right of \eqref{eq:f''ineq} can be estimated as follows:
\begin{align*}
\aint_{\bfS^n}u_{z(s)}^{\frac{\alpha-1}\alpha}\,d\theta&\cdot\aint_{\bfS^n} u_{z(s)}^{-\frac{1+\alpha}\alpha} \left(x\cdot\vec{a}\right)^2\,d\theta-\left(\aint_{\bfS^n}u^{-\frac1\alpha} x\cdot\vec{a}\,d\theta\right)^2\\
&=\frac12\aint_{\bfS^n}\aint_{\bfS^n} \frac{\left(u_{z(s)}(y)x\cdot\vec{a} - u_{z(s)}(x)y\cdot\vec{a}\right)^2}
{u_{z(s)}(x)^{\frac{1+\alpha}{\alpha}}u_{z(s)}(y)^{\frac{1+\alpha}{\alpha}}}\,d\theta(x)\,d\theta(y)\\
&\geq \frac12 \left(\text{diam}(\Omega)\right)^{-2\frac{1+\alpha}{\alpha}}\aint_{\bfS^n}\aint_{\bfS^n}\left(u_{z(s)}(y)x\cdot\vec{a} - u_{z(s)}(x)y\cdot\vec{a}\right)^2\,d\theta(x)\,d\theta(y)
\end{align*}
The integral on the right is now invariant under translations.  Evaluating at the center of mass where $\int_{\bfS^n}u_z(x)x\cdot\vec{a}\,d\theta(x)=0$, it equals $\frac{2}{n+1}\aint_{\bfS^n} u_z(x)^2\,d\theta(x)\geq \frac{2}{n+1}\left(\aint u_z(x)d\theta(x)\right)^2\geq 2c(n)\text{diam}(\Omega)^2$, since $\aint_{\bfS^n}u_z(x)\,d\theta$ is the mean width of $\Omega$.

Next we control the denominator on the right-hand side of \eqref{eq:f''ineq}:
If $\alpha\geq 1$ then $\aint_{\bfS^n}u_{z(s)}^{\frac{\alpha-1}{\alpha}}\,d\theta\leq \left(\text{diam}(\Omega)\right)^{\frac{\alpha-1}{\alpha}}$.  Thus in these cases we have $F''(s)\leq -2D$, where $D=\frac{c(n)}{(\operatorname{diam}(\Omega))^2}$, and hence $\mathcal{E}_\alpha(\Omega,z)\leq \mathcal{E}_\alpha(\Omega)-D|z-z_e|^2$.

In the case $0<\alpha<1$ we write
$\aint u_{z(s)}^{\frac{\alpha-1}{\alpha}}\,d\theta = \E^{-\frac{1-\alpha}{\alpha}F(s)}$, so we have either $F(s)\leq F(0)-1$ or $F''(s)\leq -2D$, where $2D = c(n)\operatorname{diam}(\Omega)^{-\frac2\alpha}\E^{\frac{2(1-\alpha)}{\alpha}\mathcal{E}_\alpha(\Omega)}\E^{\frac{-2(1-\alpha)}{\alpha}}$, and the estimate of the Lemma follows. We note that $\mathcal{E}_\alpha(\Omega)$ (and hence also $D$) is controlled above and below in terms of $\rho$, in view of Proposition \ref{rho-e}.
\end{proof}

\begin{lemma}\label{e-cont} Fix $\alpha>0$.  For $\Omega\in\Gamma_\rho$ let $z_e(\Omega)$ be the entropy point of $\Omega$.  Then ${\mathcal E}_\alpha(\Omega)$ and $z_e(\Omega)$ are both continuous functions on $\Gamma_\rho$ with respect to Hausdorff distance.
\end{lemma}

\begin{proof}
Fix $\Omega_0\in\Gamma_\rho$, and let $z_0\in\operatorname{Int}(\Omega_0)$ be the entropy point of
$\Omega_0$. We will prove continuity of the entropy and entropy point at $\Omega_0$.  Let $u^0_z$ be the support function of $\Omega_0$ about $z$, for each $z\in\operatorname{Int}(\Omega_0)$, and let $u_z$ be the support function of a neighbouring convex set $\Omega$.  Since $z_0$ is in the interior of $\Omega_0$, there is
$\eta>0$ such that $u^0_{z_0}(x)\geq\eta$ for all $x\in\bfS^n$.  We may assume that $\eta D\le 1$. It follows that for $|z-z_0|<\frac{\eta}{2}$ we have $u^0_z(x)\geq \frac\eta2$ for all $x\in\bfS^n$.  Now assume that $\sigma=d_{\mathcal H}(\Omega,\Omega_0)<\frac{\eta}{4}$.  Then $\sup_{\bfS^n}|u(x)-u_0(x)|\leq \sigma$, so we have $|u_z(x)-u^0_z(x)|\leq\sigma\leq \frac{2\sigma}{\eta}u^0_z(x)$ for each $x\in\bfS^n$ and each $z$ with $|z-z_0|<\frac\eta{2}$.  Since $\E^{\mathcal{E}_\alpha}$ is homogeneous of degree one in $u$, it follows that for such $z$ we have
$$
\mathcal{E}_\alpha(\Omega_0,z)+\log\left(1-\frac{2\sigma}{\eta}\right)\leq \mathcal{E}_\alpha(\Omega,z)\leq \mathcal{E}_\alpha(\Omega_0,z)+\log\left(1+\frac{2\sigma}{\eta}\right).
$$
In particular, we have $\mathcal{E}_\alpha(\Omega,z_0)\geq \mathcal{E}_\alpha(\Omega_0)+\log(1-\frac{2\sigma}{\eta})$.  On the other hand, by Lemma \ref{e-stable}, provided $\sigma<\frac{\eta}{2}\tanh(\frac{D\eta}{4})$, for any $z$ with $|z-z_0|^2=\frac1D\log\left(\frac{\eta+2\sigma}{\eta-2\sigma}\right)$ we have $D|z-z_0|^2<1$ and so
$$
\mathcal{E}_\alpha(\Omega,z)\leq \mathcal{E}_\alpha(\Omega_0)-D|z-z_0|^2+\log(1+\frac{2\sigma}{\eta})\leq\mathcal{E}_\alpha(\Omega_0)
+\log(1-\frac{2\sigma}{\eta})\leq\mathcal{E}_\alpha(\Omega,z_0).
$$
Since $\mathcal{E}_\alpha(\Omega,z)$ is concave in $z$ by Lemma \ref{e-stable}, it follows that the maximum of $\mathcal{E}_{\alpha}(\Omega,z)$ occurs in the ball $B_{\frac1D\log\left(\frac{\eta+2\sigma}{\eta-2\sigma}\right)}(z_0)$, so we have
$\left|z_e(\Omega)-z_0\right|\leq \sqrt{\frac1D\log\left(\frac{\eta+2\sigma}{\eta-2\sigma}\right)}\leq \sqrt{\frac{2\sigma}{\eta D}}\to 0$ (as $\sigma\to0$), and
$$
\mathcal{E}_\alpha(\Omega_0)+\log(1-\frac{2\sigma}{\eta})\leq \mathcal{E}_\alpha(\Omega)\leq \mathcal{E}_\alpha(\Omega_0)+\log(1+\frac{2\sigma}{\eta}),
$$
so that $|\mathcal{E}_\alpha(\Omega)-\mathcal{E}_\alpha(\Omega_0)|\leq \max\left\{\log(1+\frac{2\sigma}{\eta}),\log\left(\frac{1}{1-\frac{2\sigma}{\eta}}\right)\right\}\to 0$ as $\sigma\to 0$.
\end{proof}

\begin{lemma}\label{uniform-entropy-interior}
For any $\delta>0$ there exists $\varepsilon_1>0$ such that for every $\Omega\in\Gamma_\delta$,
\begin{equation}\label{e-est}
\operatorname{dist}(z_e({\Omega}), \partial \Omega)\ge \epsilon_1.
\end{equation}
\end{lemma}

 \begin{proof}
We argue by the contradiction.  Suppose that the statement (\ref{e-est}) is not true. Then there is a sequence
$\{\Omega_k\}$ of domains in $\Gamma_\rho$ such that
\[\operatorname{dist}(z_k, \partial \Omega_k) \to 0, \quad k\to \infty.\]
By Blaschke selection theorem (cf. \cite{Sch}*{Theorem 1.8.7}, there exists a subsequence of $\{\Omega_k\}$
in $\Gamma_\delta $, which we still denote as
$\Omega_k$, converging to a convex body $\Omega_0\in\Gamma_\rho$.  By Lemma \ref{e-cont}, $z_e(\Omega_0)=\lim_{k\to\infty}z_e(\Omega_k)\in\partial\Omega_0$.  This contradicts Lemma \ref{e-p}.
\end{proof}

 Now we are ready to prove Theorem \ref{C0}.

\begin{proof}[Proof of Theorem \ref{C0}]

Note that under \eqref{gcf-alpha-nor} for $\alpha>\frac{1}{n+2}$, we have for all $t\geq 0$ that $|\Omega_t|=|B(1)|$, and (by Proposition \ref{rho-e} and Corollary \ref{non-collapsing}) there exists $\rho>0$ such that $\Omega_t\in\Gamma_\rho$ for every $t\geq 0$.

By (\ref{eq:entropy-nor}) we have that
$$
\mathcal{E}_\alpha(\Omega_t)-\mathcal{E}_\alpha^\infty \le \mathcal{E}_\alpha(\Omega_t, 0)-\mathcal{E}_\alpha^\infty \le 0.
$$
This implies that $\lim_{ t\to \infty} \mathcal{E}_\alpha(\Omega_t, 0)=\mathcal{E}_\alpha^\infty$ and $\lim_{t\to \infty}\left( \mathcal{E}_\alpha(\Omega_t, 0)-\mathcal{E}_\alpha(\Omega_t)\right)=0$. Let $z_e(\Omega(t))$ be the entropy point of $\Omega(t)$. By Lemma (\ref{e-stable}) we have that when $\mathcal{E}_\alpha(\Omega_t,)>\mathcal{E}_\alpha(\Omega_t)-1$,
$$
|z_e(\Omega(t))-0|^2 \le \frac{1}{D} \left|\mathcal{E}_\alpha(\Omega_t, 0)-\mathcal{E}_\alpha(\Omega_t)\right|
$$
which approaches to zero as $t\to \infty$. The claimed result then follows from Lemma \ref{uniform-entropy-interior}.
 \end{proof}

\begin{corollary}\label{C0-sum} Let $u(x, t)$ be as in Theorem \ref{C0}. Then there exists $\Lambda=\Lambda(\Omega_0, \alpha,  n)>0$ such that
\begin{equation}\label{C0-2sides}
\frac{1}{\Lambda}\le u(x, t)\le \Lambda.
\end{equation}
\end{corollary}

\begin{proof}
The upper bound is immediate since the diameter of $\Omega_t$ is bounded by Proposition \ref{rho-e}.  The lower bound for $t\geq T_0$ is provided by Theorem \ref{C0}, and for $t<T_0$ we use the fact that $\tilde u(x,\tau) = u(x,t)\E^{-t}$ is non-increasing in $\tau$, hence in $t$, so we have
$$
u(x,t)\geq \E^{t-T_0}u(x,T_0)\geq \varepsilon\E^{-T_0}.
$$
\end{proof}

\section{$C^2$-estimates }

 In this section we derive uniform $C^2$-estimates from the $C^0$-estimate (\ref{C0-2sides}).
The first is a upper estimate on the Gauss curvature, which was first proved by Tso \cite{Tso} for the case $\alpha=1$ (see also Hamilton \cite{Hamilton-gauss}) and by the first author in \cite{Andrews-pjm} (Theorem 6) for all other $\alpha>0$:

\begin{theorem}\label{C2-u} Suppose $u(x, t)$ is the solution of (\ref{eq-gcf-s-nor}) with initial data $u(x, 0)=u_0(x)$, where $u_0(x)$ is the support function of $\Omega_0$ with $|\Omega_0|=|B(1)|$ and $\Omega_0\in\Gamma_\rho$ for some $\rho\in(0,1)$.  Then there exists  a constant $C=C(n, \alpha,\rho)>0$ such that
\begin{equation}\label{upper-c2}
K(x, t)\le C\min\left\{\sup_{M_0}K, 1+t^{-\frac{n}{1+n\alpha}}\right\}.
\end{equation}
\end{theorem}

\begin{proof} See \cite{Andrews-pjm}*{Theorem 6}.
\end{proof}

Our crucial new contribution is the following lower bound on the Gauss curvature, which crucially uses the lower bound on the support function from Theorem \ref{C0}:

\begin{theorem}\label{thm-lower-k} Suppose $u(x, t)>0$ is a positive solution of (\ref{eq-gcf-s-nor}), obtained from the un-normalized flow (\ref{gcf-alpha-nor}), with initial data $u(x,0)=u_0(x)$, where $u_0(x)>0$ is the support function of $\Omega_0$ with $|\Omega_0|=|B(1)|$. Then there exists a constant $\epsilon_2=\epsilon_2(n, \Omega_0)>0$ such that
\begin{equation}\label{lower-C2}
K^\alpha(x, t)\ge \epsilon_2.
\end{equation}
\end{theorem}

\begin{proof}  The result can be proved by a similar line of argument as in the proof of Theorem  5.2 of \cite{Guan-N}. We provide a different argument here.  We observe that the minimum of $\tilde K$ over $\tilde M_\tau$ is non-decreasing, so we have positive lower bounds on $K$ for any finite time, given by $K(x,t)\geq \left(\inf_{M_0}K\right)\E^{-nt}$.  Thus it suffices to obtain a uniform lower bound for large $t$.

The key result we will use is the following estimate for the normalized flow (see \cites{Smoczyk, AMZ}):

\begin{lemma}\label{lemma-lower-k}
If $\tilde u$ evolves according to \eqref{eq-gcf-s}, then for any $\tau_2>\tau_1$ and any $x\in\bfS^n$ we have
$$
\tilde K^\alpha(x,\tau_2)\geq \frac{\tilde u(x,\tau_1)-\tilde u(x,\tau_2)}{(1+n\alpha)(\tau_2-\tau_1)}.
$$
\end{lemma}

\begin{proof}
See \cite{AMZ}*{Theorem 14}.\end{proof}

To apply this, we first show that for any time $t$ corresponding to an un-rescaled time $\tau_1$, we can choose a suitable $\tau_2>\tau_1$ so that the numerator is uniformly positive, using the bounds from Corollary \ref{C0-2sides}:  Since $\lambda\leq u(x,t) \leq \Lambda$ for all $x$ and $t$, we have for \eqref{eq-gcf-s} the following:
$$
\lambda\E^{-t}\leq \tilde u(x,\tau(t)) \leq \Lambda\E^{-t}.
$$
It follows that if we choose $\tau_1=\tau(t)$ and $\tau_2 = \tau(t+\log(\frac{2\Lambda}{\lambda}))$ then
$$
\tilde u(x,\tau_2)\leq \Lambda\E^{-t}\frac{\lambda}{2\Lambda} = \frac{\lambda}{2}\E^{-t}\leq \tilde u(x,\tau_1)-\frac{\lambda}{2}\E^{-t}.
$$
We also observe that $\tau_2-\tau_1$ is no greater than the  extinction time of $\tilde\Omega_{\tau_1}$.  By comparing with an enclosing sphere of radius $r_+=\Lambda\E^{-t}$, we find the time to extinction is no greater than $\frac{r_+^{1+n\alpha}}{1+n\alpha} = \frac{1}{1+n\alpha}\Lambda^{1+n\alpha}\E^{-(1+n\alpha)t}$.  This gives the estimate
\begin{align*}
K^\alpha\left(x,t+\log\left(\frac{2\Lambda}{\lambda}\right)\right) &= \left(\frac{\lambda}{2\Lambda}\right)^{n\alpha}\E^{-n\alpha t}\tilde K^\alpha(x,\tau_2)\\
&\geq \left(\frac{\lambda}{2\Lambda}\right)^{n\alpha}\E^{-n\alpha t}\frac{\tilde u(x,\tau_1)-\tilde u(x,\tau_2)}{(1+n\alpha)(\tau_2-\tau_1)}\\
&\geq \left(\frac{\lambda}{2\Lambda}\right)^{n\alpha}\E^{-n\alpha t}\frac{\frac{\lambda}{2}\E^{-t}}{\Lambda^{1+n\alpha}\E^{-(1+n\alpha)t}}\\
&=\frac{\lambda^{1+n\alpha}}{(2\Lambda)^{1+2n\alpha}}.
\end{align*}
Since $t\geq 0$ is arbitrary, this provides a uniform lower bound on $K^\alpha$ for all sufficiently large times, and the Theorem is proved.
\end{proof}

  There exists however another proof of Theorem \ref{thm-lower-k} which uses neither  Lemma \ref{lemma-lower-k} nor the differential Harnack estimate. Instead when $\alpha \ne 1$ it needs a simple lemma comparing the entropies. The proof is more self-contained. The following quantity was introduced in \cite{Andrews-pjm} (for the normalized flow (\ref{eq-gcf-s-nor})) generalizing the entropy introduced in \cite{Chow2}:  \[\mathcal{Z}_\alpha (t)\doteqdot \left(\aint_{\mathbb{S}^n} K^{\alpha -1}(x, t)\, d\theta(x)\right)^{\frac{1}{\alpha-1}}.\]
 \begin{lemma}\label{comparison-entropies}
Let $\Omega$ be a convex body with $V(\Omega)=V(B(1))$. For any $\alpha>0$ we have that
$$
\mathcal{Z}_\alpha (\Omega) \ge e^{\mathcal{E}_\alpha(\Omega)}.
$$
 The equality holds if and only if $\lambda u=K^\alpha$ for $\lambda=\aint_{\mathbb{S}^n} K^{\alpha-1}\, d\theta$.
\end{lemma}
\begin{proof}  Without the loss of the generality we may assume that $\mathcal{E}_\alpha(\Omega)$ is attained at $z_0=0$. We simply denote the support function $u_{z_0}(x)$ with respect $z_0$ as $u(x)$. The condition $V(\Omega)=V(B(1))$ yields that $\aint_{\mathbb{S}^n} \frac{u}{K}\,d \theta =1$. Hence viewing $d\sigma=\frac{u}{K} \frac{d\theta}{\omega_n}$ as a probability measure, for $\alpha\in (0, 1)$,
\begin{eqnarray*}
\aint_{\mathbb{S}^n} K^{\alpha -1}\, d\, \theta &=& \aint_{\mathbb{S}^n} \frac{K^\alpha}{u}\, \frac{u}{K}\, d\theta =\aint_{\mathbb{S}^n} \frac{K^\alpha}{u} d \sigma\\
&\le&\left( \aint_{\mathbb{S}^n}\frac{K}{u^{\frac{1}{\alpha}}} d\sigma\right)^{\alpha}= \left( \aint_{\mathbb{S}^n}\frac{K}{u^{\frac{1}{\alpha}}} \, \frac{u}{K}\, d\theta \right)^{\alpha}.
\end{eqnarray*}
The claimed result follows from the above easily. Equality holds if and only if $\frac{K^\alpha}{u}=\lambda$, a constant, which can be determined by $\aint_{\mathbb{S}^n} \frac{u}{K}\, d\theta=1$. For $\alpha>1$,
\begin{eqnarray*}
\aint_{\mathbb{S}^n} K^{\alpha -1}\, d\, \theta &=& \aint_{\mathbb{S}^n} \frac{K^\alpha}{u}\, \frac{u}{K}\, d\theta \ge \left( \aint_{\mathbb{S}^n}\frac{K}{u^{\frac{1}{\alpha}}} \, \frac{u}{K}\, d\theta \right)^{\alpha}.
\end{eqnarray*}
The result follows by taking $\frac{1}{\alpha-1}$ power on the both side. For $\alpha=1$, it can be obtained by taking the limit.
\end{proof}

 Now by Corollary \ref{key} we can conclude that $\mathcal{Z}_\alpha (\Omega)\ge 1$ if the volume is normalized. This is all we need on $\mathcal{Z}_\alpha(\Omega)$ (in particular the monotonicity of $\mathcal{Z}_\alpha(t)$ is not needed).

{\it Alternate proof of Theorem \ref{thm-lower-k}}. Consider the solution $u(x, t)$ of the normalized flow (\ref{eq-gcf-s-nor}). By Corollary \ref{C0-sum} and  Theorem \ref{C2-u} we may assume that $\frac1\Lambda \le u\le \Lambda $ and $K\le \Lambda$ for some $\Lambda>0$. Set $\eta(t)=(\mathcal{Z}_\alpha(t))^{\alpha-1} $, $A=(u_{ij}+u\bar{g}_{ij})$, and $\mathcal{L} =\alpha \frac{\dot{\sigma}_n^{ij}(A)\bar{\nabla}_i\bar{\nabla}_j}{\eta\sigma_n^{\alpha+1}(A)}$ with $\sigma_n(A)$ being the $n$-th elementary symmetric function of $A$ with $\sigma_n(A)=K^{-1}$. Here recall the notations from \cite{Guan-N} with $A$ being the inverse of the second fundamental form and $\bar{\nabla}$ being the covariant derivative of $\mathbb{S}^n$ and $\bar{g}$ being the round metric.  Since $u$ satisfies (\ref{eq-gcf-s-nor}),
\[K^{\alpha}=\eta (u-u_t).\] Using this equation,  direct computations yield
\begin{equation}\label{eq-alpha-kalpha-ev} \left(\frac{\partial}{\partial t}-\mathcal{L}\right)K^{\alpha}=-n\alpha K^{\alpha} +\frac{\alpha \sigma_{n-1}(A)K^{2\alpha}}{\eta \sigma_n(A)},\end{equation}
\begin{equation}\label{eq-alpha-u-ev} \left(\frac{\partial}{\partial t}-\mathcal{L}\right) u =-\frac{1+n\alpha}{\eta} K^{\alpha} +u+\frac{\alpha u\sigma_{n-1}(A)K^{\alpha}}{\eta \sigma_n(A)}.\end{equation}
At the mean time observe  the formula that for any $l\ge 0$
\begin{eqnarray*}
\left(\frac{\partial}{\partial t}-\mathcal{L}\right)\log(fg^l)&=& \frac{1}{f} \left(\frac{\partial}{\partial t}-\mathcal{L}\right)f +\frac{l}{g}\left(\frac{\partial}{\partial t}-\mathcal{L}\right)g+\alpha \frac{\dot{\sigma}_n^{ij}(A)}{\sigma_n^{\alpha+1}(A)}\frac{\bar{\nabla}_i f\bar{\nabla}_j f}{f^2}\\
&\quad&+l\, \alpha   \frac{\dot{\sigma}_n^{ij}(A)}{\sigma_n^{\alpha+1}(A)}\frac{\bar{\nabla}_i g\bar{\nabla}_j g}{g^2}\\
&\ge& \frac{1}{f} \left(\frac{\partial}{\partial t}-\mathcal{L}\right)f +\frac{l}{g}\left(\frac{\partial}{\partial t}-\mathcal{L}\right)g.
\end{eqnarray*}
Combining with (\ref{eq-alpha-kalpha-ev}) and (\ref{eq-alpha-u-ev}) we  have the estimate
\begin{equation}\label{eq-lower-key-curv}
\left(\frac{\partial}{\partial t}-\mathcal{L}\right)\log(K^\alpha u^l)\ge (l-n\alpha)-\frac{l (1+n\alpha)}{\eta(t)}\frac{K^\alpha}{u}.
\end{equation}
Note that for $\alpha\ge 1$, $\eta(t)\ge 1$ and for $\alpha<1$, $\eta(t)\ge \Lambda^{\alpha-1}$.
Now let $m(t)=\min_{\mathbb{S}^n}  K^\alpha(\cdot, t) u^l(\cdot, t)$ for $l=n\alpha +1$. Clearly by Corollary \ref{C0-sum} for estimating $K^\alpha$ from the below it suffices to obtain a lower estimate on $m(t)$.
 Set $\lambda=\frac{100}{(n\alpha+1)^2 \Lambda^{n\alpha+3-\alpha}}$. By enlarging $\Lambda$ we may assume that $m(0)\le \frac{\lambda}{2}$. We prove below by contradiction that $m(t)\ge \lambda$ for all $t\ge 0$. Assume the contrary, and let $t_0$ be the first time when $m(t_0)$ touches $\lambda$.
  If $x(t_0)$ is where the minimum $m(t_0)$ is attained for $K^\alpha(\cdot, t) u^{1+n\alpha}(\cdot, t)$ then (\ref{eq-lower-key-curv}) implies at $t=t_0$ with $l=n\alpha +1$
\begin{eqnarray*}
0&\ge &\frac{d}{dt} \log m(t)\ge 1-(1+n\alpha)^2 \frac{K^\alpha(x(t), t)}{\Lambda^{\alpha-1}u(x(t), t)}\\
&=& 1-\frac{(1+n\alpha)^2}{\Lambda^{\alpha-1}} \frac{K^\alpha(x(t), t)u^{l}(x(t), t)}{u^{l+1}(x(t), t)}\\
&\ge& 1-(1+n\alpha)^2 \Lambda^{l+2-\alpha} m(t).
\end{eqnarray*}
Here we have used that $u\ge \frac{1}{\Lambda}$. The above implies that $m(t_0)\le \frac{\lambda}{100}$. But this is a contradiction since $m(t_0)=\lambda$.

 Once we have the two sided estimates of $K^\alpha$, the proof of Theorem 10 in \cite{Andrews-pjm} gives the following estimate on the second fundamental forms of $M_t$.

\begin{theorem}\label{c2-sharp}
 Suppose $u(x, t)>0$ is the solution of (\ref{gcf-alpha-nor}) with initial data $u(x, 0)=u_0(x)$, where $u_0(x)>0$ is the support function of $\Omega_0$ with $|\Omega_0|=|B(1)|$.
There exists  a constant $C>0$, depending on $n, \Omega_0$ such that
\begin{equation}\label{upperc2}
\operatorname{trace}\left(\bar{\nabla}_i \bar{\nabla}_j u+u\delta_{ij}\right) \le C.
\end{equation}
Moreover the symmetric tensor $A$ has the lower estimate:
\begin{equation}\label{positivity}
\bar{\nabla}_i \bar{\nabla}_j u+u\bar{g}_{ij}\ge \frac{1}{C} \bar{g}_{ij}.
\end{equation}
\end{theorem}

Combining Proposition \ref{rho-e}, Corollary \ref{non-collapsing}, Theorem \ref{mono1}, Theorem  \ref{C0}, Theorem \ref{C2-u} and Theorem \ref{c2-sharp},  we conclude that there exists a positive constant $C$ depending only on the initial data such that for the unique positive solution to (\ref{eq-gcf-s})
\begin{equation}\label{C2-norm}
\|u(\cdot, t)\|_{C^2(\mathbb S^n)} \le C.
\end{equation}

\section{Convergence to solitons}

Since (\ref{eq-gcf-s}) is a concave parabolic equation, by   Krylov's theorem \cite{kry} and the standard theory on the parabolic  equations, estimates (\ref{C2-norm}) and (\ref{positivity})  imply the bounds on all derivatives (space and time) of $u(x, t)$. More precisely, for any $k\ge 3$, there exists $C_k\ge 0$, depending only on the initial value such that for $t\ge 1$
\begin{equation}\label{Ck}
\|u(\cdot, t)\|_{C^k(\mathbb S^n)}\le C_k.
\end{equation}
Now for any $T>0$ and sequence $\{t_j\}\to \infty$, consider $u_j(x, t)\doteqdot u(x, t-t_j)$. We have the following result on the sequential convergence.

\begin{proposition}\label{s-con} After passing to a subsequence, on $\bfS^n \times [-T, T]$, $\{u_j\}$ converges in the $C^\infty$-topology to a smooth function $u_\infty(x)$ which is a  self-similar solution to (\ref{eq-gcf-s}). Namely it satisfies the equation
$$
\lambda u_\infty(x)=K^\alpha_{\infty}(x)
$$
where $\lambda=\aint_{\mathbb{S}^n}K^{\alpha-1}_\infty(x)\, d\theta(x)$.
\end{proposition}
\begin{proof} For each $k\in\mathbb{N}$, let $u_j(x,t):=u(x,t+j)$.  Then $u_j$ is a solution of \eqref{eq-gcf-s-nor} for each $j$, and we have bounds in $C^k$ for every $k$, independent of $j$.  It follows that $u_j:\ \bfS^n\times[0,1]\to\bfR$ converges (for a subsequence of $j$) in $C^\infty$ to a limit $u^\infty$ which is again a solution of \eqref{eq-gcf-s-nor}.  Furthermore, if we denote by $\Omega^j_t=\Omega_{t+j}$ the corresponding convex body, then we have ${\mathcal E}_\alpha(\Omega^j_t) = \mathcal{E}_\alpha(\Omega_{t+j})\to\mathcal{E}_\alpha^\infty$ for every $t$, so we have that ${\mathcal E}_\alpha(\Omega^\infty_t)$ is constant.  It follows from Theorem \ref{mono1} that $u^\infty$ is a soliton:  The function $f = \frac{K^\alpha}{u}$ is constant, so that
$$
K^\alpha_\infty(x, t)=c(t)u_\infty(x, t)
$$
for some constant $c(t)$. Since $\aint_{\mathbb{S}^n}\frac{u_\infty}{K_{\infty}}=1$, we deduce that $c(t)=\aint_{\mathbb{S}^n} K^{\alpha-1}_\infty $. But now \eqref{eq-gcf-s-nor} gives that $\frac{\partial u}{\partial t}=0$, so that $u$ is a stationary solution
\end{proof}

Theorem 2 of \cite{Andrews-IMRN}, together with the previous proposition, implies the  following result.

\begin{theorem}\label{c-infty}
The flow (\ref{gcf-alpha-nor}) converges in $C^\infty$-topology to a smooth soliton $u_\infty$ ($M_\infty$) which has $K>0$ and satisfies the soliton equation:
\begin{equation} \label{eq-soliton}
\lambda  u \cdot \left(\det (u\operatorname{id} +\bar{\nabla}^2 u)\right)^{\alpha}=1.
\end{equation}
Here $\lambda=\aint_{\mathbb{S}^n} \left(\det (u\operatorname{id} +\bar{\nabla}^2 u)\right)^{1-\alpha}\, d\theta(x)$.
\end{theorem}

\section{Convergence to spheres in the centrally symmetric case}

It remains an interesting question whether or not the round sphere (ball) is the unique compact soliton for $\alpha>\frac{1}{n+2}$. For $\alpha=\frac{1}{n+2}$, it was proved by Calabi \cite{Calabi} that the solitons are ellipsoids. Hence Theorem \ref{c-infty} recovers the main result of \cite{Andrews-jdg}. Since when $\alpha=\frac{1}{n}$, the soliton must be round sphere by \cite{Mcoy} Theorem \ref{c-infty} recovers the main theorem of \cite{Chow}.    For the case $\alpha\ge 1$ we have the following result for the centrally symmetric case, which generalizes the result of Firey \cite{Firey}.

\begin{proposition}\label{soliton} Assume $u$ is a soliton with associated body $\Omega$ (namely  $\lambda u=K^{\alpha}$ with $\lambda=\aint_{\mathbb{S}^n} K^{\alpha-1}$). Then the following holds.
\begin{enumerate}[label={(\roman*).}]
\item The origin is the entropy point of $\Omega$ and $|\Omega|=|B(1)|$;
 \item When $\alpha \ge 1$, the volume of $\Omega^*_0$ satisfies
\begin{equation}\label{lower-dual}
|\Omega^*_0|\ge |B(1)|.
\end{equation}
This implies that if the origin is the Santal\'o point of $\Omega$, then $\Omega$ is a ball.
\item More generally, if $\alpha\ge1$, for any $\alpha'\in [\frac{1}{n+2}, \frac{\alpha}{\alpha+1}]$
\begin{equation}\label{entropy-lower}
\mathcal{E}_{\alpha'}(\Omega, 0)\le 0.
\end{equation}
This implies that if the origin is also the entropy point of  $\mathcal{E}_{\alpha'}(\Omega)$ for some $\alpha'\in [\frac{1}{n+2}, \frac{\alpha}{\alpha+1}]$, then $\Omega$ must be a ball.
\end{enumerate}

In particular, if $\Omega$ is centrally symmetric  then $\Omega=B(1)$, and the flow (\ref{gcf-alpha-nor}) converges in $C^\infty$-topology to a ball if the initial body is centrally symmetric.
\end{proposition}
\begin{proof}
The equation clearly gives $\aint_{\mathbb{S}^n}\frac{u}{K}=1$. On the other hand $\int_{\mathbb{S}^n}\frac{u}{K}=(n+1)|\Omega|$. This proves that $|\Omega|=|B(1)|$.
By the soliton equation it is easy to check that
$$
\int_{\mathbb{S}^n} \frac{x_j}{u^{\frac{1}{\alpha}}(x)}d\, \theta(x)=0.
$$
Hence $0$ is the entropy point with respect to $\mathcal{E}_\alpha(\Omega)$.

For $\alpha>1$, we have (in the notation of Theorem \ref{mono1})
$$
\lambda = \aint K^{\alpha-1}\,d\theta =\aint f\,d\sigma \geq \left(\aint f^{\frac1\alpha}d\sigma\right)^\alpha = \E^{(\alpha-1)\mathcal{E}_\alpha(\Omega)}\geq 1,
$$
where we used $\aint d\sigma = \frac{|\Omega|}{|B(1)|}=1$, and then we argue as follows:
\begin{eqnarray*}
\frac{|\Omega^*_0|}{|B(1)|}=\aint_{\mathbb{S}^n} \frac{1}{u^{n+1}}\, d\theta &=& \lambda^{n+1}\aint_{\mathbb{S}^n} \frac{1}{K^{\alpha(n+1)}}\, d\theta\\
&\ge & \aint_{\mathbb{S}^n} \frac{1}{K^{\alpha(n+1)}}\, d\theta\\
&\ge& \left(\aint_{\mathbb{S}^n} \frac{1}{K}\, d\theta\right)^{\alpha (n+1)}\\
&\ge& 1.
\end{eqnarray*}
At the last line above we use the isoperimetric inequality $\aint_{\mathbb{S}^n}\frac{1}{K}\, d\theta(x) \ge 1$.  By the Blaschke-Santal\'o inequality, if $0$ is the Santal\'o point  we have that $|\Omega^*|\le |B(1)|$. Hence equality holds in the above inequalities, from which it is easy to see $u=1$. The proof of part (iii) is similar.
 \end{proof}

\section{Applications and stability for the entropy}

 Using the proposition we  derive some estimates on the entropy for general convex domains  which can be viewed as stability results for Corollary \ref{key}. These results  are inspired by \cite{Ivaki}. To formulate the result we recall the concept of the curvature image $\Lambda_{\alpha} \Omega$ \cite{Lutwak}, which can be defined via the  solution to certain Monge-Amper\'e equation (precisely the Minkowski problem) and the compatibility conditions (\ref{1stvar}), which hold if the origin is the entropy point. The convex body $\Lambda_\alpha \Omega$ is characterised by having the so-called surface area measure function $f_{\Lambda_\alpha \Omega}(x)$ (for the smooth case it is the reciprocal of the Gauss curvature) given by
\begin{equation}\label{eq:curvature-image}
f_{\Lambda_\alpha\Omega}(x)=\frac{|\Omega|}{|B(1)|} \E^{-\frac{\alpha-1}{\alpha}\mathcal{E}_\alpha(\Omega)}u^{-\frac{1}{\alpha}}_e(x).
\end{equation}
The (normalised) mixed volume $V_1( \Lambda_\alpha \Omega, \Omega)\doteqdot \frac{1}{n+1}\int_{\mathbb{S}^n} u_{\Omega}(x) f_{\Lambda_\alpha \Omega}(x)\, d\theta(x)$ is then given by
$$
V_1( \Lambda_\alpha \Omega, \Omega)=|\Omega|.
$$
which then implies (by the Alexandrov-Fenchel inequality) that
\begin{equation}\label{key2}
\frac{|\Omega|}{|\Lambda_\alpha \Omega|} \ge 1 \quad\mbox{ and }\quad \frac{V_1(\Omega, \Lambda_\alpha \Omega)}{|\Lambda_\alpha \Omega|} \ge 1.
\end{equation}

We first derive the estimates as a corollary of Proposition \ref{soliton}.

\begin{corollary}\label{g-BS} Let $\Omega$ be a smooth strictly convex body.  Suppose that either $\alpha=\frac1n$, or $\alpha\geq 1$ and $\Omega$ is centrally symmetric.  Then
\begin{equation}\label{eq:66}
\mathcal{E}_{\alpha}(\Omega)\ge \frac1{n+1}\log\left(\frac{|\Omega|}{|B(1)|}\right)+\frac{n}{n+1}\log\left(\frac{|\Omega|}{|\Lambda_{\alpha}\Omega|} \right).
\end{equation}
    The equality holds if and only if $\Omega$ is a round ball.
\end{corollary}
\begin{proof} The proof is essentially from \cite{Ivaki}. The key is the observation that the entropy point of $\mathcal{E}_\alpha(\Omega)$ is invariant under the flow
\begin{equation}\label{inverse-gcf}
\frac{\partial}{\partial t} X(x,t) =\frac{\langle X(x, t), \nu(x, t)\rangle^{1+\frac{1}{\alpha}}}{K(x, t)} \nu(x, t)
\end{equation}
which in terms of the support function can be written as
\begin{equation}\label{inverse-gcf-support}
\frac{\partial}{\partial t} u(x, t)=\frac{u^{1+\frac{1}{\alpha}}(x, t)}{K(x, t)}.
\end{equation}
Hence we shall assume that the origin is the entropy point of the initial convex domain $\Omega$.
Let $\Omega_t$ be the evolving convex domain. Here $u(x, t)$ and $K(x, t)$ denote the support function and the Gauss curvature of $\partial \Omega_t$ (with respect to the origin). We denote by $u_{\Lambda_\alpha \Omega}$ and $K_{\Lambda_\alpha \Omega}$ the support function and the Gauss curvature of $\partial (\Lambda_\alpha \Omega)$.
The evolution equation of the following three quantities $\mathcal{J}_i$ ($1\le i\le 3$) along the flow (\ref{inverse-gcf-support}) holds the key to the proof. A straight forward computation yields
\begin{equation}\label{eq-igcf-1}
\frac{d}{dt} \mathcal{J}_1(t)=(n+1)\frac{\int_{\mathbb{S}^n} f^{-\frac1\alpha}\, d\sigma}{\int_{\mathbb{S}^n} \, d\sigma}, \mbox{  with  }\mathcal{J}_1(t)=\log \left(\frac{|\Omega_t|}{|B(1)|}\right),
\end{equation}
where as before we write $f=\frac{K^\alpha}{u}$ and $d\sigma = \frac{u}{K}d\theta$.  We also have
\begin{equation}\label{eq-igcf-2}
\frac{d}{dt} \mathcal{J}_2(t)=\frac{\int_{\mathbb{S}^n}\, d\sigma}{\int_{\bfS^n}f^{\frac1\alpha}\,d\sigma}, \mbox{  with  }\mathcal{J}_2(t)=\mathcal{E}_\alpha(\Omega_t).
\end{equation}

Combining the above we have that the scaling invariant quantity
$\mathcal{Q}(t)=\mathcal{E}_\alpha(\Omega_t)-\frac{1}{n+1}\log \left(\frac{|\Omega_t|}{|B(1)|}\right)$, which by Corollary \ref{key} is always bounded below by $0$,  satisfies the equation:
\begin{eqnarray}\label{key-igcf-3}
\frac{d}{dt}\mathcal{Q}(t)&=& \mathcal{J}_2'-\frac{1}{n+1}\mathcal{J}_1' \nonumber \\
&=&\left(\frac{\int_{\mathbb{S}^n}\, d\sigma}{\int_{\bfS^n}f^{\frac1\alpha}\,d\sigma} -\frac{\int_{\mathbb{S}^n} f^{-\frac1\alpha}\, d\sigma)}{\int_{\mathbb{S}^n} \, d\sigma}\right)\le 0.
\end{eqnarray}
This together with Proposition \ref{rho-e} controls the support function from above and below under the evolution, and the curvature can be estimated above and below following the methods of \cite{Ivaki} (or by the method of Section 5).  It follows that the solution exists for finite time and expands to infinity under \eqref{inverse-gcf}, and that after rescaling to fixed volume the solutions converge smoothly (for a subsequence of times) to a soliton.

To obtain the result of the theorem we derive the evolution equation of $\log |\Lambda_\alpha \Omega_t|$. The following computation has been carried in \cite{Ivaki}. Note that $K_{\Lambda_\alpha \Omega_t}^{-1}(x,t) = f_{\Lambda_\alpha\Omega}(x)$ is given by \eqref{eq:curvature-image}, and
$$
|\Lambda_\alpha \Omega_t|=\frac{1}{n+1}\int_{\mathbb{S}^n} \frac{u_{\Lambda_\alpha \Omega}(x, t)}{K_{\Lambda_\alpha \Omega_t}(x, t)}\, d\theta(x), \quad \frac{d}{dt}|\Lambda_\alpha \Omega_t|=\frac{1}{n} \int_{\mathbb{S}^n}u_{\Lambda_\alpha \Omega}(x, t) \frac{\partial}{\partial t} K_{\Lambda_\alpha \Omega_t}^{-1}(x,t).
$$
Hence we have
\begin{equation}\label{eq:igcf31}
\frac{d}{dt}\log |\Lambda_\alpha \Omega_t|=\frac{n+1}{n} \mathcal{J}_1'-(n+1)\frac{1-\frac{1}{\alpha}}{n}\mathcal{J}_2' -\frac{n+1}{n} \frac{1}{\alpha} \frac{V_1(\Omega_t, \Lambda_\alpha \Omega_t)}{|\Lambda_\alpha \Omega_t|} \mathcal{J}_2'.
\end{equation}
Next we consider the scaling invariant quantity $\mathcal{J}_3(t)=\frac{n}{n+1}\log \left(\frac{|\Omega_t|}{|\Lambda_\alpha \Omega_t|}\right)$ and its evolution equation:
\begin{eqnarray}
\frac{d}{dt} \mathcal{J}_3 &=& -\frac{1}{n+1}\mathcal{J}_1'+(1-\frac{1}{\alpha})\mathcal{J}_2'+\frac{1}{\alpha} \frac{V_1(\Omega_t, \Lambda_\alpha \Omega_t)}{|\Lambda_\alpha \Omega_t|} \mathcal{J}_2'\\
&\ge& -\frac{1}{n+1}\mathcal{J}_1'+\mathcal{J}_2'=\frac{d}{dt}\mathcal{Q}. \nonumber
\end{eqnarray}
Thus $\mathcal{Q}-\mathcal{J}_3$ is non-increasing, and
the claimed estimate follows from the above and the classification of solitons provided by Firey \cite{Firey} for the case $\alpha=1$, by Chow \cite{Chow} for the case $\alpha=\frac{1}{n}$ and Proposition \ref{soliton} for the case $\alpha\ge 1$:  These imply that the limiting soliton is a ball, in which case $\mathcal{Q}=\mathcal{J}_3=0$, so necessarily $\mathcal{Q}\geq \mathcal{J}_3$ initially. \end{proof}

The central symmetry assumption and the condition $\alpha\geq 1$ apppear in the above proof only in the classification of solitons, and so the inequality holds whenever it can be established that solitons are spheres.  In particular, our generalised conjecture would imply the inequality for all $\alpha>\frac{1}{n+2}$ without any central symmetry assumption.  Indeed, in the next section we provide a different argument which establishes this inequality without using the flow.

\section{Entropy stability via isoperimetric inequalities}

Now we present a result which contains Corollary \ref{g-BS} as a special case without assuming the central symmetry.
To present this more general result we extend the definition of the entropy to $\alpha<0$ by by adopting the definition \eqref{eq:def-e-z} without change, and modifying the definition \eqref{eq:def-e} by taking an infimum rather than a supremum for $\alpha<0$.  As before there is a unique entropy point in the interior of the domain in the case $\alpha<0$ (the proof of Lemma \ref{e-p} applies without change).  The result of Corollary \ref{key} (which used the Blaschke-Santal\'o inequality) gives that $\mathcal{E}_\alpha(\Omega)\geq 0$ whenever $|\Omega|=|B(1)|$, and this result can easily be extended to $\alpha<0$ using an isoperimetric inequality:  We have by the H\"older inequality for $\alpha<0$ that
$$
\E^{\mathcal{E}_\alpha(\Omega,z)} = \left(\aint_{\bfS^n}u^{1-\frac1\alpha}\,d\theta\right)^{\frac{1}{1-\frac1\alpha}}
\geq \aint_{\bfS^n} u\,d\theta = \frac{1}{|B(1)|}V_1(B, \Omega)\geq \left(\frac{|\Omega|}{|B(1)|}\right)^{\frac1{n+1}}=1,
$$
where the last inequality is the Minkowski inequality relating mean width and volume (see \cite{Sch}*{Theorem 7.2.1}), for which equality holds if and only if $\Omega$ is a ball.

We recall the affine isoperimetric inequality: For any convex body $\Omega'$, if $f_{\Omega'}$ is its surface area measure, one may define the affine surface area by
$$
\mathcal{A}(\Omega')\doteqdot \int_{\mathbb{S}^n} f_{\Omega'}^{\frac{n+1}{n+2}}(x)\, d\theta(x).
$$
The affine isoperimetric inequality relates this to the volume (see \cites{Lutwak,Sch}):

\begin{theorem}\label{affine-iso} For any convex body $\Omega$,
\begin{equation}\label{lut1}
\mathcal{A}(\Omega)^{n+2}\le (n+1)^{n+2} |B(1)|^2 |\Omega|^{n}.
\end{equation}
\end{theorem}

We now proceed to the main result:

\begin{theorem}\label{entropy-p} For any convex body $\Omega$, for $\alpha\geq\frac{1}{n+2}$,
\begin{equation}\label{eq:72}
\mathcal{E}_{\alpha}(\Omega)\ge \frac1{n+1}\log\left(\frac{|\Omega|}{|B(1)|}\right)+\frac{n}{n+1}\log\left(\frac{|\Omega|}{|\Lambda_{\alpha}\Omega|} \right).
\end{equation}
For $\alpha<0$,
\begin{equation}\label{eq:reverse}
\mathcal{E}_{\alpha}(\Omega)\le \frac1{n+1}\log\left(\frac{|\Omega|}{|B(1)|}\right)+\frac{n}{n+1}\log\left(\frac{|\Omega|}{|\Lambda_{\alpha}\Omega|} \right).
\end{equation}
The equality holds if and only if $\Omega$ is a round ball, unless $\alpha =\frac{1}{n+2}$, in which case equality also holds for ellipsoids.
\end{theorem}
\begin{proof} Without the loss of generality we may always assume that the origin is the entropy point. Recall that the surface area measure of $\Lambda_\alpha \Omega$ is given by equation \eqref{eq:curvature-image}.
%

For the case $\alpha \geq \frac{1}{n+2}$, we proceed as follows:
$$\mathcal{A}(\Lambda_\alpha \Omega)^{n+2}=\omega_n^{n+2}\left(\frac{|\Omega|}{|B(1)|}\right)^{n+1}\E^{-\frac{\alpha-1}{\alpha}\mathcal{E}_\alpha(\Omega)}\left(\aint_{\mathbb{S}^n} \left(\frac{1}{u(x)}\right)^{\frac{1}{\alpha}\frac{n+1}{n+2}}\, d\theta(x)\right)^{n+2}.
$$
Since $\alpha\geq \frac{1}{n+2}$ we have $\frac{1}{\alpha}\frac{n+1}{n+2}\geq \frac{1-\alpha}\alpha$, and the H\"older inequality gives
$$
\left(\aint_{\mathbb{S}^n} \left(\frac{1}{u(x)}\right)^{\frac{1}{\alpha}\frac{n+1}{n+2}}\, d\theta(x)\right)^{n+2}
\geq \left(\aint_{\bfS^n}\left(\frac{1}{u(x)}\right)^{\frac{1-\alpha}{\alpha}}\,d\theta(x)\right)^{\frac{n+1}{1-\alpha}}
=\E^{-\frac{n+1}{\alpha}\mathcal{E}_\alpha(\Omega)}.
$$
Note that the case $\alpha=1$ follows as a limit.
Hence we have that
\begin{equation}\label{eq:74}
\mathcal{A}(\Lambda_\alpha \Omega)^{n+2}\ge \omega_n^{n+2}\left(\frac{|\Omega|}{|B(1)|}\right)^{n+1} \E^{-(n+1)\mathcal{E}_\alpha(\Omega)}.
\end{equation}
The claimed result follows by combining the above with the affine isoperimetric inequality (\ref{lut1}) for the body $\Lambda_\alpha\Omega$.

For the case $\alpha<0$ we apply instead the isoperimetric inequality,   namely
\begin{equation}\label{isoperim-ineq}
A^{n+1}(\Lambda_\alpha \Omega)\ge \left(\frac{|\Lambda_\alpha \Omega|}{|B(1)|}\right)^n \omega_n^{n+1},
\end{equation}
where
$A(\Omega')$ is the surface area of $\partial \Omega'$ for any convex body $\Omega'$.

On the other hand, the H\"older inequality implies that
\begin{eqnarray}\label{eq:A-ineq}
A^{n+1}(\Lambda_\alpha\Omega)&=&\left(\frac{|\Omega|}{|B(1)|}\right)^{n+1}\E^{-\frac{\alpha-1}{\alpha}\mathcal{E}_\alpha(\Omega)}
\left(\int_{\mathbb{S}^n} u^{-\frac{1}{\alpha}}(x)\, d\theta(x)\right)^{n+1}\notag\\
&\le& \left(\frac{|\Omega|}{|B(1)|}\right)^{n+1} \E^{-(n+1)\mathcal{E}_\alpha(\Omega)}\omega_n^{n+1}.
\end{eqnarray}
Putting \eqref{isoperim-ineq} and \eqref{eq:A-ineq} together we have (\ref{eq:reverse}).  \end{proof}

Special cases of Theorem \ref{entropy-p} were proved for $n=1$ in \cite{Ivaki} using the flow method.

\bigskip

\noindent {\it Acknowledgements.}  The first author acknowledges the hospitality of the Yau Mathematical Sciences Center at Tsinghua University where some of this research was carried out.


\end{document}